\documentclass[11.05pt,a4paper]{amsart}
\usepackage{amssymb,amsmath,amsfonts,amscd,enumerate,verbatim,calc,amsthm}
\usepackage{latexsym}
\usepackage{amsthm,amsfonts,amssymb,mathrsfs}
\usepackage{rotating}
\usepackage[leqno]{amsmath}
\usepackage{xspace}
\usepackage[all]{xy}
\usepackage{longtable}
\usepackage[colorlinks=true]{hyperref}
\usepackage[all]{xy}
\input xy
\xyoption{all}

\headsep=1cm \numberwithin{equation}{section}
\newtheorem{thm}{Theorem}[section]
\newtheorem{cor}[thm]{Corollary}
\newtheorem{lem}[thm]{Lemma}
\newtheorem{prop}[thm]{Proposition}
\newtheorem{defn}[thm]{Definition}

\newtheorem{rem}[thm]{Remark}
\newtheorem{ques}[thm]{Question}

\newtheorem{Con*}[thm]{Conjectuer}

\newcommand{\Ann}{\mbox{Ann}\,}

\newcommand{\rank}{\mbox{rank}}
\newcommand{\uhom}{{\mathbf R}\Hom}
\newcommand{\utp}{\otimes^{\mathbf L}}

\renewcommand{\H}{\mbox{H}}

\newcommand{\D}{\mbox{D}}

\newcommand{\fa}{\mathfrak{a}}

\newcommand{\fm}{\mathfrak{m}}
\newcommand{\fp}{\mathfrak{p}}

\newcommand{\C}{\mbox{C}}

\newcommand{\vdim}{\mbox{vdim}\,}

\newcommand{\amp}{\mbox{amp}\,}

\def\id{\operatorname{\mathsf{id}}}

\def\Gpd{\operatorname{\mathsf{Gpd}}}
\def\Gd{\operatorname{\mathsf{G-dim}}}

\def\GCpd{\operatorname{\mathsf{G_C-pd}}}
\def\GCd{\operatorname{\mathsf{G_C-dim}}}
\def\GCad{\operatorname{\mathsf{G_{\uhom_R(C,D)}-dim}}}
\def\GDd{\operatorname{\mathsf{G_D-dim}}}

\def\pd{\operatorname{\mathsf{pd}}}

\def\amp{\operatorname{\mathsf{amp}}}

\def\Ext{\operatorname{\mathsf{Ext}}}
\def\gr{\operatorname{\mathsf{grade}}}
\def\depth{\operatorname{\mathsf{depth}}}
\def\Hom{\operatorname{\mathsf{Hom}}}
\def\dim{\operatorname{\mathsf{dim}}}
\def\r{\operatorname{\mathsf{r}}}
\def\Tor{\operatorname{\mathsf{Tor}}}

\DeclareMathOperator{\Supp}{Supp}
\DeclareMathOperator{\Spec}{Spec}

\begin{document}

\title[characterizations of a dualizing complex]
{some characterizations of dualizing complexes in terms of $G_{C}$-dimension}

\author[M. Rahro Zargar]{Majid Rahro Zargar}

%\address{K. Divaani-Aazar, Department of Mathematics, Alzahra University, Vanak, Post Code
%19834, Tehran, Iran-and-School of Mathematics, Institute for Research in Fundamental Sciences
%(IPM), P.O. Box 19395-5746, Tehran, Iran.}
%\email{kdivaani@ipm.ir}

\address{Majid Rahro Zargar, Department of Engineering Sciences, Faculty of Advanced Tech-
nologies, University of Mohaghegh Ardabili, Namin, Ardabil, Iran, and}
\address{School of Mathematics, Institute for Research in Fundamental Sciences (IPM), P.O. Box: 19395-5746, Tehran, Iran.}
\email{zargar9077@gmail.com}
\email{m.zargar@uma.ac.ir}

\begin{abstract}Let $(R,\fm)$ be a local ring and $C$ be a homologically bounded and finitely generated $R$-complex. Then, we prove that $C$ is a dualizing complex of $R$ if and only if $C$ is a Cohen-Macaulay semidualizing complex of type one or $\mu_R^{\inf C+\dim_R(C) }(\fm,R)=\beta_{\inf C}^R(C)$. Also, we show that a semidualizing complex $C$ is dualizing if and only if there exists a type one Cohen-Macaulay $R$-module of finite $G_{C}$-dimension or there exists a type one Cohen-Macaulay $R$-complex of finite $G_{C}$-dimension such that $\dim_R(X)=\dim_R(C)-\gr_C(X)$. Furthermore, for a semidualizing $R$-complex $C$, we prove that $C\sim R$ if and only if there exists a type one Cohen-Macaulay $R$-module $M$ which belongs to the Auslander class $\mathcal{A}_C(R)$.
\end{abstract}
\maketitle

%\tableofcontents

\section{Introduction}
The theory of dualizing complexes of Grothendieck and Hartshorne (\cite{RH}, chapter v) has turned out to be a useful tool even in commutative algebra. For instance, Peskine and Szpiro used dualizing complexes in their (partial) solution of Bass's conjecture concerning finitely generated modules of finite injective dimension over a Noetherian local ring (\cite{Pes}, Chapitre I, §5). Foxby \cite{Fax}, Golod \cite{Golod}, and Vasconcelos \cite{Vas} independently initiated the study of semidualizing modules (under different names). To unify the study of dualizing complexes and semidualizing modules, as a generalization of the concept of semidualizing modules, Christensen, in \cite{CF1} introduced the notion of semidualizing complexes in the derived category of $R$-modules and was able to provide some interesting results and to describe a procedure for constructing Cohen-Macaulay local rings with any finite number of semidualizing modules. Immediate examples of semidualizing complexes are the shifts of the underlying rings and dualizing complexes. These are precisely semidualizing complexes with finite projective dimensions and semidualizing complexes with finite injective dimensions.

The Gorenstein dimension or $G$-dimension, for finitely generated modules, was introduced by Auslander in \cite{Aus} and it was shown that it is a finer invariant than the projective dimension and that satisfies equality of the Auslander-Buchsbaum type. As a generalization of this concept, Golod \cite{Golod} used a semidualizing $R$-module $C$ to define $G_C$-dimension for finitely generated modules. In \cite{SY}, Yassemi studied Gorenstein dimension for complexes through consistent use of the functor \emph{$\uhom_R(-,R)$} and the related category $\mathcal{R}(R)$. Then, in \cite {CF1}, Christensen studied the functor $\uhom_R(-, C)$ and the related category $\mathcal{R}(C)$ for a semidualizing complex $C$ and introduced the concept of $G_C$-dimension sharing the nice properties of Auslander's classical $G$-dimension.

Throughout this paper, $R$ is a commutative Noetherian ring with a nonzero identity. In the case where $(R,\fm)$ is a local ring with the unique maximal ideal $\fm$, the residue field of $R$ is denoted by $k$ and the $\fm$-adic completion of $R$ will be denoted by $\widehat{R}$. The Peskine-Szpiro intersection theorem is one of the main results in commutative algebra. This theorem has many interesting corollaries, one of the well-known of them is the following theorem: Let $R$ be a commutative Noetherian local ring. Suppose that there exists a Cohen-Macaulay R-module of finite projective dimension. Then, the local ring R is Cohen-Macaulay; (see \cite[Proposition 6.2.4]{PR}). So, based on this theorem, we are naturally led to the following conjecture:
\begin{Con*}
Let $R$ be a commutative Noetherian local ring. Suppose that there
exists a Cohen-Macaulay $R$-module $M$ of finite $G$-dimension. Then the local ring
$R$ is Cohen-Macaulay.
\end{Con*}
If this conjecture is proved, as the assertion is itself very interesting, its proof might give another easier proof of the Peskine-Szpiro intersection theorem. However, nothing has been known about this conjecture until now. Regarding the above conjecture, Takahashi, in \cite{Tak}, provided several conditions in terms of $G$-dimension which is equivalent to the Gorensteinness of $R$. Indeed, he provided the following result:
\begin{thm} Let $(R,\fm)$ be a local ring. Then the statements are equivalent:
\begin{itemize}
\item[(i)]{{$R$ is Gorenstein.}}
\item[(ii)]{{$R$ admits an ideal $I$ of finite $G$-dimension such that the factor ring $R/I$ is Gorenstein.}}
\item[(iii)]{{$R$ admits a Cohen-Macaulay module of type 1 and finite $G$-dimension.}}
\item[(iv)]{{$R$ is a local ring of type 1 admitting a Cohen-Macaulay module of finite $G$-dimension.}}
 \end{itemize}
\end{thm}
This theorem says that the above conjecture is true if the type of $R$ or $M$ is one. In this direction, the principal aim of this paper is to provide some characterizations of dualizing complexes in terms of some conditions. For example, one of our aims is to provide a generalization of the above result in the category of $R$-complexes. Indeed, we characterize a dualizing complex (dualizing module) in terms of the existence of a Cohen-Macaulay $R$-complex (Cohen-Macaulay $R$-module) of type one with finite $G_{C}$-dimension. The organization of this paper is as follows: In section 2, we collect some definitions and notions in the derived category of $R$-modules which will be used in the thorough of this paper. In section 3, among other things, we provide some characterizations of a dualizing complex. First, in Theorem \ref{Anni}, we prove that over a local ring $(R,\fm)$ an $R$-complex \emph{$C\in\mathcal{\D}_{\Box}^f(R)$} is a Cohen-Macaulay semidualizing of type 1 if and only if it is a dualizing complex. Next, in Theorem \ref{ddd}, the following result provides a characterization of a dualizing complex of $R$ in terms of the certain Bass number of $R$.
\begin{thm}Let $(R,\fm)$ be a local ring and let $C$ be a semidualizing $R$-complex. Then the following statements are equivalent:
\begin{itemize}
\item[(i)]{$C$ is dualizing}.
\item[(ii)]{$\mu_R^{\inf C+\dim C }(\fm,R)=\beta_{\inf C}^R(C)$}.
\end{itemize}
\end{thm}
Indeed, the above result is a generalization of the result of Foxby and Roberts in \cite[Corollary 9.6.3 and Remark 9.6.4]{BH}, which have shown that a local ring $(R,\fm)$ is Gorenstein if and only if $\mu_{R}^{\dim R}(\fm, R)=1$.

In Theorem \ref{Maa}, as another main result, we prove a generalization of Takahashi's result \cite[Theorem 2.3]{Tak} which is mentioned above. Indeed, we have the following theorem:
\begin{thm}Let $(R,\fm)$ be a local ring and let $C$ be a semidualizing $R$-complex with $\r_R(C)=1$. Assume that $R$ admits a Cohen-Macaulay $R$-module $M$ of finite $G_{C}$-dimension, then $C$ is dualizing.
\end{thm}
Next, as a consequence of the above theorem, Proposition \ref{Ma} and Theorem \ref{Anni}, we prove the following result in Theorem \ref{Maaaa}.
\begin{thm}Let $(R,\fm)$ be a local ring and let $C$ be a semidualizing $R$-complex and consider the following statements:
\begin{itemize}
\item[(i)]{$C$ is dualizing}.
\item[(ii)]{$R$ admits a Cohen-Macaulay complex of type 1 with finite $G_{C}$-dimension.}
\item[(iii)]{$\r_R(R)=\beta_{\inf C}^R(C)$ and there exists a Cohen-Macaulay $R$-complex $X$ of finite $G_{C}$-dimension.}
\item[(iv)]{$\r_R(C)=1$ and there exists a Cohen-Macaulay $R$-complex $X$ of finite $G_{C}$-dimension.}
\end{itemize}
Then the implications \emph{(i)$\Longrightarrow$(ii), (ii)$\Longrightarrow$(iii) and (iii)$\Longleftrightarrow$(iv)} hold and the implication \emph{(iv)$\Longrightarrow$(i)} hods in the each of the following cases:
\begin{itemize}
\item[(1)]{$\amp X=0$.}
\item[(2)]{$X$ is a dualizing complex.}
\item[(3)]{ $X$ satisfies the condition $\dim_R(X)=\dim_R(C)-\gr_C(X)$.}
\end{itemize}
\end{thm}

Finally, in Theorem \ref{fai}, we provide the following result:
\begin{thm}Let $(R,\fm, k)$ be local ring and $C$ be a semidualizing $R$-complex. Then the following statements are equivalent:
 \begin{itemize}
\item[(i)]{$C\sim R$}.
\item[(ii)]{$k\in\mathcal{A}_C(R)$}.
\item[(iii)]There exists a Cohen-Macaulay $R$-module $M$ of type one such that $M\in\mathcal{A}_C(R)$.
\end{itemize}
\end{thm}
This result can be considered as an improvement of \cite[Proposition 8.3]{CF1} which says that over a local ring $(R,\fm, k)$ a semidualizing $R$-complex $C$ is isomorphic to a shift of $R$ if and only if the residue field $k$ belongs to the Auslander class $\mathcal{A}_C(R)$.

\section{Prerequisites}
$\mathbf{Complexes.}$
 An $R$ complex $X$ is a sequence of $R$-modules $X_i$ and $R$-linear
maps $\partial_{i}^{X}: X_i\rightarrow X_{i-1}, i\in\mathbb{Z}$. The module $X_i$ is called the module in degree $i$, and
$\partial_{i}^{X}$ is the $i$-th \textit{differential} ; composition of two consecutive differentials always yields
the zero map, i.e., $\partial_{i-1}^{X}\partial_{i}^{X}=0$. If $X_{i}=0$ for $i\neq0$ we identify $X$ with the module
in degree $0$, and an $R$-module $M$ is thought of as a complex $0\rightarrow M\rightarrow0$, with $M$
in degree $0$. For any integer $n$, the $n$-fold shift of a complex $(X,\xi^X)$ is the complex $\Sigma^nX$ given by $(\Sigma^n X)_v=X_{v-n}$ and $\xi_{v}^{\Sigma^nX}=(-1)^n\xi_{v-n}^{X}$. The homological position and size of a complex are captured by the numbers
\textit{supremum}, \textit{infimum} and \textit{amplitude}
defined by $$\sup X:=\sup \{i\in \mathbb{Z}~|~\H_i(X)\neq 0\},$$ $$\inf X:=\inf \{i\in \mathbb{Z}~|~
\H_i(X) \neq 0\}, \text{and}$$
$$\amp X:=\sup X-\inf X.$$
 With the usual convention that $\sup \emptyset=-\infty$ and $\inf \emptyset=\infty$.

$\mathbf{Derived}$ $\mathbf{Category.}$ The \textit{derived category} of the category of $R$-modules is
the category of $R$-complexes localized at the class of all quasi-isomorphisms (see
\cite{RHart}), is denoted by $\mathrm{D}(R)$. We use the symbol $\simeq$ for denoting
isomorphisms in $\mathrm{D}(R)$. The symbol $\thicksim$ indicates isomorphism up to shift. The
full subcategory of homologically left (resp. right) bounded complexes is denoted by $\mathrm{D}_{\sqsubset}(R)$ (resp.
$\mathrm{D}_{\sqsupset}(R)$). Also, we denote the full subcategory of complexes with finitely generated homology
modules that are homologically bounded (resp. homologically left bounded) by $\mathrm{D}_{\Box}^f(R)$ (resp.
$\mathrm{D}_{\sqsubset}^f(R)$).

Let $\textbf{X},\textbf{Y}\in \mathrm{D}(R)$. The left-derived
tensor product complex of $\textbf{X}$ and $\textbf{Y}$ in $\mathrm{D}(R)$ is denoted by $\textbf{X}\otimes_
R^{{\bf L}}\textbf{Y}$ and is defined by $$\textbf{X}\otimes_R^{{\bf L}}\textbf{Y}\simeq \textbf{F}\otimes_R
\textbf{Y}\simeq \textbf{X} \otimes_R \textbf{F}^{'}\simeq \textbf{F}\otimes_R\textbf{F}^{'},$$ where $\textbf{F}$
and $\textbf{F}^{'}$ are semi-flat resolutions of $\textbf{X}$ and $\textbf{Y}$, respectively. Also, the right derived homomorphism complex
of $\textbf{X}$ and $\textbf{Y}$ in $\mathrm{D}(R)$ is denoted by ${\bf R}\Hom_R(\textbf{X},\textbf{Y})$ and is
defined by $${\bf R}\Hom_R(\textbf{X},\textbf{Y}) \simeq \Hom_R(\textbf{P},\textbf{Y})\simeq \Hom_R(\textbf{X},
\textbf{I})\simeq \Hom_R(\textbf{P},\textbf{I}) ,$$ where $\textbf{P}$ and $\textbf{I}$ are semi projective resolution
of $\textbf{X}$ and semi-injective resolution of $\textbf{Y}$, respectively. For
$i\in\mathbb{Z}$ we set $\Tor^R_i(X,Y)=\H_i(X\utp_R Y)$ and $\Ext^i_{R}(X,Y)=\H_{-i}(\uhom_{R}(X,Y ))$. For modules
$X$ and $Y$,  this agrees with the notation of classical homological algebra, so no confusion arises. For $R$-complexes $M$, $X$ and an integer $n\in \mathbb{Z}$ there is the following identity of complexes:
$$\uhom_R(M,\Sigma^n X) = \Sigma^n\uhom_R(M,X),and$$
$$\uhom_R(\Sigma^n M,X) = \Sigma^{-n}\uhom_R(M,X).$$

$\mathbf{Numerical}$ $\mathbf{Invariants.}$ Let $(R,\fm,k)$ be a local ring with residue field $k$. The \textit{depth} of an $R$-complex $X$ is defined by
$\depth_R(X) =-\sup\uhom_{R}(k,X)$;
and the Krull dimension of $X$ is defined as follows:
 \[\begin{array}{rl}
\dim_R(X)&=\sup\{\dim R/\fp-\inf X_{\fp}~|~ \fp\in\Spec R\}\\
&=\sup\{\dim R/\fp-\inf X_{\fp}~|~ \fp\in\Supp_R X\},
\end{array}\]
where $\Supp_R X = \{\fp\in\Spec R~|~X_\fp \not\simeq 0\}=\bigcup_{i\in\mathbb{Z}}\Supp_R \H_i(X)$. Note that for
modules these notions agree with the standard ones. For an $R$-complex $X$ in $\mathrm{D}_{\sqsubset}^f(R)$ such that $X\not\simeq 0$ one always has the inequality $\depth_R(X)\leq\dim_R(X)$, and a homologically bounded and finitely generated $R$-complex $X$ is said to be \textit{Cohen-Macaulay} whenever the equality $\depth_R(X)=\dim_R(X)$ holds. For an $R$-complex $X$ and integer $m\in\mathbb{Z}$, the $m$-th \textit{Bass number} of $X$ is defined by $\mu_R^m(\fm,X)=\rank_{k}\H_{-m}(\uhom_R(k,X))=\rank_{k}\Ext^m_{R}(k,X)$, and for a homologically bounded $R$-complex $X$, the \textit{type} of $X$ is denoted by $\r_R(X)$ and defined by $\r_R(X)=\mu_R^{\depth_R(X)}(\fm,X)$. Also, the $m$-th \textit{Betti number} of $X$ is defined by $\beta_m^{R}(X)=\rank_{k}\H_{m}(k\utp_{R}X)=\rank_{k}\Tor^R_{m}(k,X)$.

\begin{defn} \emph{An $R$-complex $C$ is said to be \textit{semidualizing} for $R$ if and only
if $C\in\mathcal{\D}_{\Box}^f(R)$ and the homothety morphism $\chi_{C}^{R}:R\longrightarrow\uhom_{R}(C,C)$ is an isomorphism. Notice that in the case where $C$ is an $R$-module, this notion coincides with the concept that $\Ext_{R}^i(C, C)=0$ for all $i>0$ and the homothety morphism $\chi_{C}^{R}:R\longrightarrow\Hom_{R}(C, C)$ is an isomorphism. An $R$-complex $\D\in\mathcal{\D}_{\Box}^f(R)$ is said to be a \textit{Dualizing} complex if and only if it is a semidualizing complex with finite injective dimension.}
\end{defn}
%%%%%%%%%%%%%%%%%%%%%%%%%%%%%%%%%%%%%%%%%%%%%%%%%%%%%%%%%%%%%%%%%%%%%%%%%%%%%%%%%%%%%%%%%%%%%%%%%%%%%%%%%%%%%%%%%%%%%%%%%%%%%%%%%%%%%%%%%%%%%%%%%%%%%%%%%%%%%%%%%%%%
\begin{defn}\label{2.2}\emph{Let $C$ be a semidualizing complex for $R$. An $R$-complex $X$ in $\mathcal{\D}(R)$ is said to be $C$-\textit{reflexive} if and only if $X$ and $\uhom_{R}(X,C)$ belong to $\mathcal{\D}_{\Box}^f(R)$ and the biduality morphism $\delta_{X}^C:X\longrightarrow\uhom_R(\uhom_R(X,C),C)$ is invertible. The class of complexes in $\mathcal{\D}_{\Box}^f(R)$ which are $C$-reflexive is denoted by $\mathcal{R}(C)$, in view of \cite[Definition 3.11]{CF1}, for an $R$-complex $X$ in $\mathcal{\D}_{\Box}^f(R)$ we define the $G$-dimension of $X$ with respect to $C$ as follows:
\[ \GCd_{R}(X)=\begin{cases}
       \inf C-\inf\uhom_R(X,C)& \text{if $X\in\mathcal{R}(C)$}\\
       \infty& \text{if $X\notin\mathcal{R}(C)$}
       \end{cases} \]
}
\end{defn}
\begin{defn}\label{2.3}\emph{Let $C$ be a semidualizing $R$-module. An $R$-module $M$ is said to be $C$-\textit{Gorenstein projective} if:
\begin{itemize}
\item[(i)]{{$\Ext^{\geq 1}_{R}(M,C\otimes_R P)=0$ for all projective $R$-modules $P$.}}
\item[(ii)]{{There exist projective $R$-modules $P_0, P_1,\dots$ together with an exact sequence:
$$0\longrightarrow M \longrightarrow C\otimes_R P_0\longrightarrow C\otimes_R P_1\longrightarrow\dots,$$
and furthermore, this sequence stays exact when we apply to it the functor
$\Hom_R(-,C\otimes_{R}Q)$ for any projective $R$-module $Q$.}}
\end{itemize}For a homologically right-bounded $R$-complex $X$ we define:
$$\GCpd_R(X)=\inf_{Y}\{\sup\{n\in\mathbb{Z}| Y_n\neq 0\}\},$$
where the infimum is taken over all $C$-Gorenstein projective resolutions $Y$ of $X$, which in view of \cite[Examples 2.8(b)(c)]{HOJ}, they always exist. Note that when $C=R$ this notion recovers the concept of
ordinary Gorenstein projective dimension. }
\end{defn}
Next, we recall the concept of Gorenstein dimension with respect to a semidualizing $R$-module $C$, which was originally introduced by Golod \cite{Golod}.
\begin{defn}\label{DEf}\emph{The $G_{C}$-class, denoted by $G_{C}(R)$, is the collection of finite $R$-module $M$ such that
\begin{itemize}
\item[(i)]{{$\Ext^i_{R}(M,C)=0$ for all $i>0$.}}
\item[(ii)]{{$\Ext^i_{R}(\Hom_{R}(M,C),C)=0$ for all $i>0$.}}
\item[(iii)]{{The canonical map $M\longrightarrow\Hom_{R}(\Hom_{R}(M,C),C)$ is an isomorphism.}}
\end{itemize}
For a non-negative integer $n$, the $R$-module $M$ is said to be of
$G_{C}$-dimension at most $n$, if and only if there exists an exact
sequence
$$0\longrightarrow G_n\longrightarrow G_{n-1}\longrightarrow\cdots \longrightarrow G_0\longrightarrow M\longrightarrow0,$$
where {$G_{i}\in G_{C}(R)$} for $0\leq i\leq n$. If such a sequence does not exist, then we write {$\Gd_{C}(N)=\infty$}.
Note that when $C=R$ this notion recovers the concept of Gorenstein dimension which was introduced in \cite{Aus}.}
\end{defn}
%%%%%%%%%%%%%%%%%%%%%%%%%%%%%%%%%%%%%%%%%%%%%%%%%%%%%%%%%%%%%%%%%%%%%%%%%%%%%%%%%%%%%%%%%%%%%%%%%%%%%%%%%%%%%%%%%%%%%%%%%%%%%%%%%%%%%%%%%%%%%%%%%%%%%%%%%%%%%%%%
\begin{rem}\emph{Notice that if $C$ is a semidualizing $R$-module and $X\in\mathcal{\D}_{\Box}^f(R)$, then, by \cite[Proposition 3.1]{HOJ}, the Definitions \ref{2.2} and \ref{2.3} are coincide, that is, $\GCd_{R}(X)=\GCpd_R (X)$. By, \cite[Theorem 2.16]{HOJ}, for a homologically
right-bounded complex $Z$ of $R$-modules there is the equality $\GCpd_R (Z)=\Gpd_{R\ltimes C}(Z)$. Also, by \cite[Theorem 2.7]{SY}, for a finitely generated $R$-module $M$ and a semidualizing $R$-module $C$, we have that the two Definitions \ref{2.2} and \ref{DEf} are coincide, that is, $\Gd_{C}(M)=\GCd_{R}(M)$.}
\end{rem}
\begin{defn}\emph{Let $C$ be a semidualizing complex for $R$. The {\it{Auslander class}} of $R$ with respect to $C$, $\mathcal{A}_C(R)$, is the full subcategories of $\mathrm{D}_{\Box}(R)$ defined by specifying their objects as follows: $X$ belongs to $\mathcal{A}_C(R)$ if and only if $C\utp_R X\in\mathrm{D}_{\Box}(R)$ and the canonical map $\gamma_X^C: X\rightarrow\uhom_R(C,C\utp_R X)$ is an isomorphism.}
\end{defn}
It is straightforward to check that $R\in\mathcal{A}_C(R)$ and, obviously, $\mathcal{A}_R(R)=\mathrm{D}_{\Box}(R)$. On the other hand, the Auslander class behaves as expected under completion. Indeed, by \cite[Proposition 5.8]{CF1}, one has $X\in\mathcal{A}_C(R)$ if and only if $X\otimes_R\widehat{R}\in\mathcal{A}_{C\otimes_R\widehat{R}}(\widehat{R})$.
%%%%%%%%%%%%%%%%%%%%%%%%%%%%%%%%%%%%%%%%%%%%%%%%%%%%%%%%%%%%%%%%%%%%%%%%%%%%%%%%%%%%%%%%%%%%%%%%%%%%%%%%%%%%%%%%%%%%%%%%%%%%%%%%%%%%%%%%%%%%%%%%%%%%%%%%%%%%%%%%

\section{results}

The starting point of this section is the following lemma, which is used in the proof of some next results. For the proof of the first lemma, the reader is referred to \cite{CF1}.
\begin{lem}{\label{Lem2.2}}Let $(R,\fm,k)$ be a local ring, $C$ be a semidualizing $R$-complex and $X$ be a homologically bounded and finitely generated $R$-complex. Then the following statements hold:
\begin{itemize}
\item[(1)]{{If $\GCd_R(X)<\infty$, then $\Gd_C(X)=\depth R- \depth_R(X).$}}
\item[(2)]{The following conditions are equivalent:}
\begin{itemize}
\item[(i)]{{$C$ is dualizing.}}
\item[(ii)]{{$\GCd_R(k)<\infty$.}}
\end{itemize}
\item[(3)]{{If $\GCd_R(X)<\infty$, then $\inf\uhom_R(X,C)=\depth_R(X)-\depth_R(C).$}}
\end{itemize}
\end{lem}
%%%%%%%%%%%%%%%%%%%%%%%%%%%%%%%%%%%%%%%%%%%%%%%%%%%%%%%%%%%%%%%%%%%%%%%%%%%%%%%%%%%%%%%%%%%%%%%%%%%%%%%%%%%%%%%%%%%%%%%%%%%%%%%%%%%%%%%%%%%%%%%%%%%%%%%%%%%%%%%%%%
%%%%%%%%%%%%%%%%%%%%%%%%%%%%%%%%%%%%%%%%%%%%%%%%%%%%%%%%%%%%%%%%%%%%%%%%%%%%%%%%%%%%%%%%%%%%%%%%%%%%%%%%%%%%%%%%%%%%%%%%%%%%%%%%%%%%%%%%%%%%%%%%%%%%%%%%%%%%%%%%%%
The following lemma will play an essential role in the proof of Theorem \ref{Maa}, which is one of our main results in the present paper.
\begin{lem}{\label{Lemma}}Let $C$ be a semidualizing $R$-complex and $M$ be a finitely generated $R$-module. Suppose that $\mathbf{x}=x_1,\dots,x_n$ is an $M$-sequence and $\GCd_R(M)$ is finite. Then $\GCd_{R}({M}/{\mathbf{x}M})=\GCd_R(M)+n$.
\end{lem}
\begin{proof}By use of induction on $n$, it is sufficient to prove in the case where $n=1$. So, let $x$ be a nonzero divisor on $M$ and consider the exact sequence
$0\longrightarrow M\stackrel{x}\longrightarrow M\longrightarrow M/xM\longrightarrow0$ which yields the following exact sequence of $R$-complexes
$$0\longrightarrow P^{'}\longrightarrow P\longrightarrow P^{''}\longrightarrow0,$$
where $P^{'}$, $P$ and $P^{''}$ are projective resolutions of $M$, $M$ and $M/xM$, respectively. Now, let $I$ be a semi-injective resolution for $C$ and consider the following exact sequence of $R$-complexes and the commutative diagram of complexes with exact rows:
$$0\longrightarrow \Hom_R(P^{''},I)\longrightarrow \Hom_R(P,I)\longrightarrow \Hom_R(P^{'},I)\longrightarrow0$$
and
$$\begin{CD}
0@>>>P^{'} @>>> P @>>> P^{''} @>>>0\\
@. @V\delta_{P^{'}}^{I}VV @VV\delta_{P}^{I}V @VV\delta_{P^{''}}^{I}V \\
0@>>> \Hom_R(\Hom_R(P^{'},I),I) @>>>\Hom_R(\Hom_R(P,I),I) @>>> \Hom_R(\Hom_R(P^{''},I),I) @>>> 0. \\
\end{CD}$$

Now, when we pass to homology, the above exact sequence and diagram yield the following long exact sequence of homologies and the long exact ladder:
$$\cdots\longrightarrow \H_i(\Hom_R(P^{''},I))\longrightarrow \H_i(\Hom_R(P,I))\longrightarrow \H_i(\Hom_R(P^{'},I))\longrightarrow \H_{i-1}(\Hom_R(P^{''},I))\longrightarrow\cdots$$
and
$$\begin{CD}
\cdots @>>>\H_i(P^{'}) @>>> \H_i(P) @>>> \H_{i}(P^{''}) @>>>\H_{i-1}(P^{'}) @>>>\cdots\\
@. @V\H_{i}(\delta_{P^{'}}^{I})VV @VV\H_i(\delta_{P}^{I})V @VV\H_i(\delta_{P^{''}}^{I})V@ VV\H_{i-1}(\delta_{P^{'}}^{I})V \\
\cdots @>>> \H_{i}(X^{'}) @>>>\H_i(X) @>>> \H_i(X^{''}) @>>>\H_{i-1}(X^{'})@>>> \cdots ,
\end{CD}$$
where $X^{''}=\Hom_R(\Hom_R(P^{''},I),I)$, $X^{'}=\Hom_R(\Hom_R(P^{'},I),I))$ and $X=\Hom_R(\Hom_R(P,I),I))$. Notice that, by our assumption, one has  $\H_{i-1}(\delta_{P^{'}}^{I})$ and $\H_i(\delta_{P}^{I})$ are isomorphisms for all $i$ and hence $\H_i(\delta_{P^{''}}^{I})$ is an isomorphism for all $i$. Also, by our assumption, $\uhom_R(M, C)$ is bounded, and so the above long exact sequence of homologies follows that $\uhom_R(M/xM, C)$ is also bounded. Therefore, $\GCd_R(M/xM)$ is finite and then by \cite[Theorem 3.14]{CF1} one has $\GCd_R(M/xM)=\depth R-\depth_R (M/xM)=\depth R-\depth_R(M)+1=\GCd_R(M)+1$, as required.
\end{proof}

%%%%%%%%%%%%%%%%%%%%%%%%%%%%%%%%%%%%%%%%%%%%%%%%%%%%%%%%%%%%%%%%%%%%%%%%%%%%%%%%%%%%%%%%%%%%%%%%%%%%%%%%%%%%%%%%%%%%%%%%%%%%%%%%%%%%%%%%%%%%%%%%%%%%%%%%%%%%%%%
%%%%%%%%%%%%%%%%%%%%%%%%%%%%%%%%%%%%%%%%%%%%%%%%%%%%%%%%%%%%%%%%%%%%%%%%%%%%%%%%%%%%%%%%%%%%%%%%%%%%%%%%%%%%%%%%%%%%%%%%%%%%%%%%%%%%%%%%%%%%%%%%%%%%%%%%%%%%%%%

Notice that, in view of \cite[Proposition 3.3,13]{BH}, over a local Cohen-Macaulay ring $R$, a finitely generated $R$-module $C$ is dualizing (canonical) if and only if it is a faithful maximal Cohen-Macaulay $R$-module of type 1, and so a semidualizing $R$-module $C$ is dualizing if and only if it is a Cohen-Macaulay $R$-module of type 1. So, the next theorem provides a generalization of this result.
\begin{thm}{\label{Anni}}Let $(R,\fm)$ be a local ring and \emph{$C\in\mathcal{\D}_{\Box}^f(R)$}. Then, $C$ is a Cohen-Macaulay semidualizing $R$-complex of type 1 if and only if $C$ is a dualizing $R$-complex.
\begin{proof}($\Longrightarrow$). First notice that $C\otimes_R \widehat{R}$ is a semidualizing complex for $\widehat{R}$ and also $C\utp_{R}\widehat{R}\simeq C\otimes_R \widehat{R}$. Hence, for all integers $i$, one has the following equalities:
 \[\begin{array}{rl}
\mu_{\widehat R}^i(\fm\widehat{R},C\otimes_R \widehat{R})&=\vdim_{k}\H_{-i}(\uhom_{\widehat{R}}(k,C\otimes_R \widehat{R}))\\
&=\vdim_{k}\H_{-i}(\uhom_{R}(k,C)\utp_{R}\widehat{R})\\
&=\vdim_{k}(\H_{-i}(\uhom_{R}(k,C))\otimes_R \widehat{R})\\
&=\vdim_{k}\H_{-i}(\uhom_{R}(k,C))\\
&=\mu_R^i(\fm,C).
\end{array}\]
Therefore, one can use \cite[Proposition 6.3.9]{CHFox} and \cite[6.2.9]{CHFox} to deduce that $\depth_{R}(C)=\depth_{\widehat{R}}(C\otimes_{R}\widehat{R})$, $\id_{\widehat{R}}(C\otimes_{R}\widehat{R})=\id_{R}(C)$ and $\dim_{R}(C)=\dim_{\widehat{R}}(C\otimes_{R}\widehat{R})$. Hence, we may and do assume that $R$ is complete, and so it has a dualizing complex as $\D$. Now, consider the following isomorphisms in the derived category $\D(R)$:
\[\begin{array}{rl}
(*)~~\uhom_{R}(k,C)&\simeq\uhom_{R}(k,\uhom_{R}(\uhom_{R}(C,\D),\D))\\
&\simeq\uhom_{R}(k\utp_R\uhom_{R}(C,\D),\D)\\
&\simeq\uhom_{R}(k\utp_R\uhom_{R}(C,\D)\utp_{k}k,\D)\\
&\simeq\uhom_{k}(k\utp_R\uhom_{R}(C,\D),\uhom_R(k,\D)),
\end{array}\]

to obtain the equality $r_{R}(C)=\underset{i+j=\depth_R(C)}\sum\beta^{R}_j(\uhom_{R}(C,\D))\mu_R^i(\fm,D)$. Next, in view of
\cite[6.2.9]{CHFox}, one has $\mu_R^i(\fm,\D)=\delta_{i\depth_R(\D)}$ which implies that $\beta^{R}_{\depth_R(C)-\depth_R(\D)}(\uhom_{R}(C,\D))=1$. On the other hand, by \cite[Corollary 3.2]{CF1} and \cite[Lemma 3.1(a)]{CF1} we have:
 \begin{equation}
 \inf\uhom_{R}(C,\D)=\depth_R(C)-\depth_R(\D)=\inf\D-\inf C,
 \end{equation}
 and
\begin{equation}
\sup\uhom_{R}(C,\D)=\dim_R(C)-\dim_R(\D).
\end{equation}
Therefore, one can use \cite[Lemma 4.3.8]{CHFox} to see that $\H_{\inf\D-\inf C}(\uhom_{R}(C,\D)$ is a principal $R$-module, and so there exists an ideal $\fa$ of $R$ such that $\H_{\inf\D-\inf C}(\uhom_{R}(C,\D))\cong R/\fa$. As $C$ and $\D$ are Cohen-Macaulay, by equalities (3.1) and (3.2), one can obtain the isomorphisms $\uhom_{R}(C,\D)\simeq\Sigma^{\inf\D-\inf C}\H_{\inf\D-\inf C}(\uhom_{R}(C,\D))\simeq\Sigma^{\inf\D-\inf C}R/\fa$. Then, we have
 \[\begin{array}{rl}
C&\simeq\uhom_{R}(\uhom_{R}(C,\D),\D)\\
&\simeq\uhom_{R}(\Sigma^{\inf\D-\inf C}R/\fa,\D)\\
&\simeq\Sigma^{\inf\C-\inf \D}(\uhom_{R}(R/\fa,\D)).
\end{array}\]

Hence, $\uhom_{R}(R/\fa,\D)\simeq\Sigma^{\inf\D-\inf C}C$, and then one has:
\[\begin{array}{rl}
\Sigma^{\inf\D-\inf C}R&\simeq\uhom_{R}(C,\Sigma^{\inf\D-\inf C}C)\\
&\simeq\uhom_{R}(C,\uhom_{R}(R/\fa,\D))\\
&\simeq\uhom_{R}(C\utp_{R}R/\fa,\D)\\
&\simeq\uhom_{R}(R/\fa,\uhom_{R}(C,\D))\\
&\simeq\uhom_{R}(R/\fa,\Sigma^{\inf\D-\inf C}R/\fa)\\
&\simeq\Sigma^{\inf\D-\inf C}\uhom_{R}(R/\fa,R/\fa).
\end{array}\]
Therefore, $R\simeq\uhom(R/\fa,R/\fa)$, and so $\fa=0$. It, therefore, concludes that $C\simeq\Sigma^{\inf\D-\inf C}\D$, and thus $C$ has a finite injective dimension, as required. The implication ($\Longleftarrow$) follows from \cite[V.3.4]{RHart} and the fact that $\mu^{\depth C}_{R}(\fm,C)\neq 0$.
\end{proof}
\end{thm}

%%%%%%%%%%%%%%%%%%%%%%%%%%%%%%%%%%%%%%%%%%%%%%%%%%%%%%%%%%%%%%%%%%%%%%%%%%%%%%%%%%%%%%%%%%%%%%%%%%%%%%%%%%%%%%%%%%%%%%%%%%%%%%%%%%%%%%%%%%%%%%%%%%%%%%%%%%%%%%%%%%%%
%%%%%%%%%%%%%%%%%%%%%%%%%%%%%%%%%%%%%%%%%%%%%%%%%%%%%%%%%%%%%%%%%%%%%%%%%%%%%%%%%%%%%%%%%%%%%%%%%%%%%%%%%%%%%%%%%%%%%%%%%%%%%%%%%%%%%%%%%%%%%%%%%%%%%%%%%%%%%%%%%%%%
The following lemma, which provides some equivalent conditions for the existence of a type one $R$-complex with finite $G_C$-dimension, will be used in the proof of some next results in the present paper.
\begin{lem}{\label{Th3}} Let $(R,\fm)$ be a local ring and $C$ be a semidualizing $R$-complex. Then, the following statements are equivalent:
\begin{itemize}
\item[(i)]{There exists an $R$-complex $Z\in\mathcal{D}_{\Box}^f(R)$ with $\GCd_{R}(Z)<\infty$ and $r_{R}(Z)=1$}.
\item[(ii)]{$r_R(R)=\beta^{R}_{\inf{C}}(C)$.}
\item[(iii)]{$r_{R}(C)=1$}.
\end{itemize}
\begin{proof}Let $Z$ be a finite $G_{C}$-dimension $R$-complex in $\mathcal{D}_{\Box}^f(R)$. Set $\GCd_{R}(Z):=s$ and notice that by \cite[Lemma 3.12]{CF1} one has $\GCd_{R}(\Sigma^{-s}Z)=\GCd_{R}(Z)-s=0$. Now, let $\depth_R(Z):=d$ and consider the following equalities:
\[\begin{array}{rl}
\H_{-d}(\uhom_{R}(k,Z))&=\H_{-d}(\uhom_{R}(k,\Sigma^{s}(\Sigma^{-s}Z)))\\
&=\H_{-d}(\Sigma^s\uhom_{R}(k,\Sigma^{-s}Z))\\
&=\H_{-(d+s)}(\uhom_{R}(k,\Sigma^{-s}Z))\\
&=\H_{-\depth R}(\uhom_{R}(k,\Sigma^{-s}Z)).\\
\end{array}\]
Notice that the last equality follows from Lemma \ref{Lem2.2}(1). Next, set $X:=\Sigma^{-s}Z$ and notice that $X^{\dag\dag}\simeq X$, where $(-)^\dag:=\uhom_R(-,C)$. Hence, using the above equalities yields the following equalities:
\[\begin{array}{rl}
r_R(Z)&=\vdim_k\H_{-\depth R}(\uhom_{R}(k,X))\\
&=\vdim_k\H_{-\depth R}(\uhom_{R}(k,X^{\dag\dag}))\\
&=\vdim_k\H_{-\depth R}(\uhom_{R}(k,\uhom_R(X^{\dag},C))\\
&=\vdim_k\H_{-\depth R}(\uhom_{R}(k\utp_R X^{\dag},C))\\
&=\vdim_k\H_{-\depth R}(\uhom_{R}(k\utp_R X^{\dag}\utp_k k,C))\\
&=\vdim_k\H_{-\depth R}(\uhom_{k}(k\utp_R X^{\dag},\uhom_R(k,C)))\\
&=\underset{\tiny{i+j=\depth R}}\sum\beta^{R}_i(X^{\dag})\mu_{R}^j(\fm,C).
\end{array}\]
On the other hand, in view of \cite[6.2.4, 6.2.9]{CHFox}, one has $\inf\{i\in\mathbb{Z} ~|~\beta_{i}^R(X^{\dag})\neq 0)\}=\inf X^{\dag}$ and $\inf\{i\in\mathbb{Z} ~|~\mu_{R}^i(\fm, C)\neq 0)\}=\depth_R(C)$. As $\GCd_{R}(X) =0$, using Lemma \ref{Lem2.2}(1)(3) and the fact that $\inf C+\depth_R(C)=\depth R$ implies that $\inf X^{\dag}=\inf C$. Therefore, one can deduce that \begin{equation}
r_R(Z)=\beta^{R}_{\inf C}(X^{\dag})\mu_{R}^{\depth_R(C)}(\fm,C).
\end{equation}

Furthermore, notice that $\GCd_R(R)$ is finite, and so $\GCd_R(R)=\depth R-\depth R=0$. Therefore, in the above argument, by replacing $Z$ with $R$, one can deduce that
\begin{equation}
r_R(R)=\beta^{R}_{\inf C}(C)\mu_{R}^{\depth_R(C)}(\fm,C).
\end{equation}

Hence, the implication (i)$\Longrightarrow$(ii) follows from the equalities $(3.3)$ and $(3.4)$. Also, the implications (ii)$\Longleftrightarrow$(iii) follows from equality $(3.4)$. Finally, the implication (iii)$\Longrightarrow$(i) is a consequence of the fact that $\GCd_R(C)$ is finite.

\end{proof}
\end{lem}

%%%%%%%%%%%%%%%%%%%%%%%%%%%%%%%%%%%%%%%%%%%%%%%%%%%%%%%%%%%%%%%%%%%%%%%%%%%%%%%%%%%%%%%%%%%%%%%%%%%%%%%%%%%%%%%%%%%%%%%%%%%%%%%%%%%%%%%%%%%%%%%%%%%%%%%%%%%%%%%%%%%%
%%%%%%%%%%%%%%%%%%%%%%%%%%%%%%%%%%%%%%%%%%%%%%%%%%%%%%%%%%%%%%%%%%%%%%%%%%%%%%%%%%%%%%%%%%%%%%%%%%%%%%%%%%%%%%%%%%%%%%%%%%%%%%%%%%%%%%%%%%%%%%%%%%%%%%%%%%%%%%%%%%%%

Foxby and Roberts, in \cite[Corollary 9.6.3 and Remark 9.6.4]{BH}, showed that a local ring $(R,\fm)$ is Gorenstein if and only if $\mu_{R}^{\dim R}(\fm,R)=1$.  So, the following result provides a characterization of a dualizing complex of $R$ in terms of the certain Bass number of $R$. Also, in view of \cite[Proposition 3.3.11]{BH}, one has if $C$ is a canonical $R$-module over a local Cohen-Macaulay ring $R$, then $\r(R)$ is equal to the minimal generator of $C$, that is, $\r(R)=\beta_{0}^R(C)$. Therefore, the following theorem shows that the converse of this result is also true.
\begin{thm}\label{ddd}Let $(R,\fm)$ be a local ring and let $C$ be a semidualizing $R$-complex. Then the following statements are equivalent:
\begin{itemize}
\item[(i)]{$C$ is dualizing}.
\item[(ii)]{$\mu^{\inf C+\dim C }(\fm,R)=\beta_{\inf C}^R(C)$}.
\end{itemize}
\end{thm}
\begin{proof}For the implication (i)$\Longrightarrow$(ii), first by Lemma \ref{Lem2.2} one has $\dim_R(C)+\inf C=\depth_R(C)+\inf C=\depth R$ because $C$ is Cohen-Macaulay and $\GCd_R(C)$ is finite. Therefore, it follows from Lemma \ref{Th3} and the fact that $r_R(C)=1$. For the implication (ii)$\Longrightarrow$(i), set $d=\inf C+\dim_R(C)$ and consider the following equalities:
\[\begin{array}{rl}
\H_{-d}(\uhom_{R}(k,R))&=\H_{-d}(\uhom_{R}(k,\uhom_{R}(C,C)))\\
&=\H_{-d}(\uhom_{R}(k\utp_R C,C)\\
&=\H_{-d}(\uhom_{R}(k\utp_R C\utp_{k}k,C))\\
&=\H_{-d}(\uhom_{k}(k\utp_R C,\uhom_{R}(k,C))).
\end{array}\]
Therefore, one has $\mu_{R}^d(\fm,R)=\underset{i+j=d}\sum\beta^{R}_i(C)\mu_{R}^j(\fm,C)$. Hence, by our assumption, one can see that
 $$\beta^{R}_{\inf C}(C)(\mu_{R}^{\dim_R(C)}(\fm,C)-1)+\underset{(i,j)\in\Lambda}\sum\beta^{R}_i(C)\mu_{R}^j(\fm,C)=0,$$
where $\Lambda=\{(i,j)\in\mathbb{Z}\times\mathbb{Z}~|~ i>\inf C, \depth_R(C)\leq j\neq\dim_R(C), i+j=d\}$. Notice that $\beta^{R}_{\inf C}(C)\neq 0$ and $\beta^{R}_i(C)\mu_{R}^j(\fm,C)$ are non-negative integers. Next, we may assume that $\mu_{R}^{\dim_R(C)}(\fm,C)\neq0$. To do this, first notice that, by a similar argument as in the proof of Theorem \ref{Anni} we may and do assume that $R$ is complete, and so $R$ has a dualizing complex as $\D$. Therefore, by use of the isomorphism (*) in the proof of Theorem \ref{Anni}, one has: $$\mu_{R}^{\dim_R(C)}(\fm,C)=\beta_{\dim_R(C)-\dim_R(D)}^R(\uhom_{R}(C,\D)).$$ Hence, it is enough to show that $\beta_{\dim_R(C)-\dim_R(D)}^R(\uhom_{R}(C,D))\neq 0$. As $\uhom_{R}(C,\D)\in\mathcal{D}_{\Box}^f(R)$, then by \cite[Lemma 1.1]{Foxby} there is a minimal semi-free resolution for $\uhom_{R}(C,\D)$ and considering the fact that $\inf\uhom_{R}(C,\D)=\depth_R(C)-\depth_R(D)$, it is enough to show that $\depth_R(C)-\depth_R(\D)\leq\dim_R( C)-\dim_R(\D)\leq\pd_R(\uhom_{R}(C,\D))$. To this end, if $\pd_R(\uhom_{R}(C,\D))=\infty$, then there is nothing to prove. Hence, assume that $\pd_R(\uhom_{R}(C,\D))<\infty$. Therefore, in view of \cite[Proposition 7.2.4]{CHFox}, one has $\id_R(C)=\id_R(\uhom(\uhom_{R}(C,\D),\D))$ is finite, and so $C$ is a dualizing complex. Therefore, we may and do assume that $\mu_{R}^{\dim C}(\fm, C)\neq0$. Hence, we have
 $$\mu_{R}^{\dim_R(C)}(\fm,C)=1 ~~~\text{and}~~~~ \underset{(i,j)\in\Lambda}\sum\beta^{R}_i(C)\mu_{R}^j(\fm,C)=0.$$ By considering a minimal semi-free resolution for $C$, one conclude that if there exists integers $(i,j)\in\Lambda$ for which $\beta^{R}_i(C)=0$, then $\pd_R(C)$ is finite and so by \cite[Theorem 8.1]{CF1}, one can see that $C\simeq R$. Hence, in view of \cite[Corollary 9.6.3 and Remark 9.6.4]{BH} there is nothing to prove. Therefore, one has $\mu_{R}^j(\fm,C)=0$ for all $j<\dim_R(C)$. Hence, $C$ is a Cohen-Macaulay $R$-complex of type 1. Therefore, one can use Theorem \ref{Anni} to complete the proof.
\end{proof}

%%%%%%%%%%%%%%%%%%%%%%%%%%%%%%%%%%%%%%%%%%%%%%%%%%%%%%%%%%%%%%%%%%%%%%%%%%%%%%%%%%%%%%%%%%%%%%%%%%%%%%%%%%%%%%%%%%%%%%%%%%%%%%%%%%%%%%%%%%%%%%%%%%%%%%%%%%%%%%%%%%%%
Here, we should notice that by the argument for the implication (ii)$\Longrightarrow$(i), in the previous theorem, we can see that if $\mu_R^{\inf C+\dim C }(\fm,R)=\beta_{\inf C}^R(C)$, then $\mu^{\dim C }(\fm,C)=1$. Also, in \cite[Corollary 1.3]{KDT}, the authors showed that over an analytically irreducible local ring the equality $\mu_R^{\dim C }(\fm,C)=1$ implies that $C$ is a dualizing complex, and so the equality $\mu_R^{\inf C+\dim C }(\fm,R)=\beta_{\inf C}^R(C)$ holds. Therefore, it is natural to ask the following question:
\begin{ques}Let $(R,\fm)$ be a local ring and $C$ be a semidualizing $R$-complex. Then, are the following conditions equivalent?
\begin{itemize}
\item[(i)]{$\mu_R^{\dim C }(\fm,C)=1$}.
\item[(ii)]{$\mu_R^{\inf C+\dim C }(\fm,R)=\beta_{\inf C}^R(C)$}.
\end{itemize}
\end{ques}

%%%%%%%%%%%%%%%%%%%%%%%%%%%%%%%%%%%%%%%%%%%%%%%%%%%%%%%%%%%%%%%%%%%%%%%%%%%%%%%%%%%%%%%%%%%%%%%%%%%%%%%%%%%%%%%%%%%%%%%%%%%%%%%%%%%%%%%%%%%%%%%%%%%%%%%%%%%%%%%%%%%%
The next theorem is one of our main results in the present paper, which is a generalization of Takahashi's result \cite[Theorem 2.3]{Tak}. Indeed, he has proven that if a local ring $R$ with $\r_R(R)=1$ has a finite $G$-dimension Cohen-Macaulay $R$-module, then $R$ is Gorenstein.
\begin{thm}{\label{Maa}}Let $(R,\fm)$ be a local ring and let $C$ be a semidualizing $R$-complex with $\r_R(C)=1$. Assume that $R$ admits a Cohen-Macaulay $R$-module $M$ of finite $G_{C}$-dimension, then $C$ is dualizing.

\begin{proof}Let $\mathbf{x}=x_1,\dots,x_n$ be a maximal $M$-sequence, then ${M}/{\mathbf{x}M}$ is an Artinian $R$-module and by Lemma \ref{Lemma} $\GCd_{R}({M}/{\mathbf{x}M})$ is finite. Therefore, by replacing $M$ with ${M}/{\mathbf{x}M}$ we may and do assume that $M$ is an Artinian $R$-module of finite $G_C$-dimension. Hence $\Ann_{R}(M)$ is an $\fm$-primary ideal, and so one can deduce that there exists a maximal $R$-sequence $\mathbf{y}=y_1,\dots,y_n$ contained in $\Ann_{R}(M)$. Therefore, in view of \cite[Proposition 6.4]{CF1} and \cite[Theorem 6.5]{CF1} we have $D:=\uhom_R({R}/{\mathbf{y}R},C)$ is a semidualizing complex for ${R}/{\mathbf{y}R}$ and $D$ is a Cohen-Macaulay ${R}/{\mathbf{y}R}$-complex if and only if $C$ is a Cohen-Macaulay $R$-complex and also $\GDd_{{R}/{\mathbf{y}R}}(M)=\GCd_R(M)+n$. Therefore, $\GDd_{{R}/{\mathbf{y}R}}(M)$ is finite. On the other hand, in view of \cite[Theorem 6.1]{CF1}, $\r_{{R}/{\mathbf{y}R}}(D)=\r_R(C)$. Therefore, Theorem \ref{Anni} implies that $C$ is a dualizing complex for $R$ if and only if $D$ is a dualizing complex for ${R}/{\mathbf{y}R}$. Hence, by replacing $R$ with ${R}/{\mathbf{y}R}$ and $C$ with $D$, we may and do assume that $\depth R=0$. Therefore, $M$ is a zero-dimensional Cohen-Macaulay $R$-module and $C$ is a semidualizing complex with $\r_R(C)=1$ and also by \cite[Theorem 3.14]{CF1} one has $\GCd_{R}(M)=\depth R-\depth_R(M)=0$. So, by Lemma \ref{Lem2.2}, $\inf C=\inf\uhom_R(M,C)$. Now, we claim that $C$ has a finite injective dimension. To do this, by \cite[Theorem 6.1.13]{CHFox}, it is enough to show that $\H_{-i}(\uhom_{R}(k,C))=0$ for all $i>-\inf C$. As, $\r_R(C)=1$, one has $\H_{-\depth_R(C)}(\uhom_{R}(k,C))\cong k$. On the other hand, \cite[Theorem 3.3(a)]{CF1} implies that $\sup\uhom_R(M,C)=\inf\uhom_R(M,C)$. Also, by \cite[Lemma 4.2.8]{CHFox} and \cite[Lemma 3.1]{CF1}, $\sup\uhom_R(M,C)\leq\sup C$ and $\inf\uhom_R(M,C)=-\depth_R(C)$. Let $\alpha$ be a non-negative integer such that $\sup\uhom_R(M,C)+\alpha=\sup C$ and so $\sup C-\alpha=-\depth_R (C)=\inf C$. Therefore, \begin{equation}
\H_{-\depth_R(C)}(\uhom_R(M,C))\neq 0 ~~\text{and}~~ \H_{i}(\uhom_R(M,C))=0 ~~\text{for all}~~ i\neq -\depth_R(C).
 \end{equation}
 Also, by \cite[Lemma 4.2.8]{CHFox}, $\sup\uhom_R(k,C)\leq\sup C$ and also $\depth_R(C)=-\sup\uhom_R(k,C)$. Hence one can deduce that $\H_{-\depth_R C}(\uhom_R(k,C))\neq 0$ and $\H_{i}(\uhom_R(k,C))=0$ for all $i>-\depth_R(C)$.

Now, let $$0=M_1 \subset M_2\subset\dots \subset M_n=M$$
be a composition series of $M$, and so we obtain the exact sequences $$0\rightarrow M_{i-1}\rightarrow M_{i} \rightarrow k\rightarrow 0,$$
for all $1\leq i\leq n$. Therefore, one can get the following exact sequence for all integers $j$ and all $1\leq i\leq n$
\begin{equation}
\H_{-j}(\uhom_R(k,C)) \rightarrow \H_{-j}(\uhom_R(M_{i},C)) \rightarrow \H_{-j}(\uhom_R(M_{i-1},C))\rightarrow \H_{-j-1}(\uhom_R(k,C)),
\end{equation}
and so we can see that $\H_{-j}(\uhom_R(M_{i},C))=0$ for all $j<\depth_R(C)$ and for all $1\leq i\leq n$. Hence, by using (3.5) and (3.6), one has the following exact sequences:

$$0\rightarrow \H_{-\depth_R(C)}(\uhom_R(k,C)) \rightarrow \H_{-\depth_R(C)}(\uhom_R(M_{i},C)) \rightarrow \H_{-\depth_R(C)}(\uhom_R(M_{i-1},C)),$$
for all $1\leq i \leq n-1$, and
$$0\rightarrow \H_{-\depth_R(C)}(\uhom_R(k,C)) \rightarrow \H_{-\depth_R(C)}(\uhom_R(M,C)) \rightarrow \H_{-\depth_R(C)}(\uhom_R(M_{n-1},C))$$ $$\rightarrow\H_{-\depth_R(C)-1}(\uhom_{R}(k,C))\rightarrow0.$$
Therefore, by setting $s=l_{R}(\H_{-\depth_R(C)-1}(\uhom_{R}(k,C)))$, one can get the following inequalities:
\[\begin{array}{rl}
l_{R}(\H_{-\depth_R(C)}(\uhom_R(M,C)))&=l_{R}(\H_{-\depth_R(C)}(\uhom_R(M_{n-1},C)))+1-s\\
&\leq l_{R}(\H_{-\depth_R(C)}(\uhom_R(M_{n-2},C)))+2-s\\
&\leq\dots\\
&\leq n-s=\l_{R}(M)-s.
\end{array}\]
As $\sup\uhom_R(M,C)=\inf\uhom_R(M,C)=-\depth_R(C)$, then $\Sigma^{t}\uhom_R(M,C)\simeq \H_{-t}(\uhom_R(M,C))$, where $t:=\depth_R(C)$. Furthermore, by \cite[Lemma 3.1]{CF1} one has $\Supp_R(M)=\Supp_R(\uhom_R(M,C))$. Hence, we have the following equalities:
\[\begin{array}{rl}
\dim_R(\H_{-t}(\uhom_R(M,C)))&=\dim_R(\Sigma^{t}\uhom_R(M,C))\\
&=\dim_R(\uhom_R(M,C))-t\\
&=\sup\{\dim R/\fp-\inf\uhom_R(M,C)_{\fp} ~|~\fp\in\Supp_R(\uhom_R(M,C))\}-t\\
&=\sup\{\dim R/\fp-\inf\uhom_R(M,C)_{\fp} ~|~\fp\in\Supp_R(M)\}-t\\
&=-\inf(\uhom_R(M,C))-t\\
&=0
\end{array}\]
Therefore, $\dim_R(\H_{-t}(\uhom_R(M,C)))=0$ and also by \cite[Lemma 3.12]{CF1} one has:
\[\begin{array}{rl}
\GCd_R(\H_{-t}(\uhom_R(M,C)))&=\GCd_R(\Sigma^{t}\uhom_R(M,C))\\
&=\GCd_R(\uhom_R(M,C))+t\\
&=\inf C-\inf\uhom_R(\uhom_R(M,C),C)+t\\
&=\inf C-\inf M+t\\
&=0
\end{array}\]
Hence, by a similar argument as in the above, on has $l_{R}(\H_{-\depth_R C}(\uhom_R(\H_{-t}(\uhom_R(M,C)),C)))\leq l_R(\H_{-t}(\uhom_R(M,C))-s$.
Therefore, one has:
\[\begin{array}{rl}
l_R(M)&=l_R(\H_0(\uhom_R(\uhom_R(M,C),C))\\
&=l_R(\H_{-t+t}(\uhom_R(\uhom_R(M,C),C))\\
&=l_R(\H_{-t}(\Sigma^{-t}\uhom_R(\uhom_R(M,C),C))\\
&=l_R(\H_{-t}(\uhom_R(\Sigma^{t}\uhom_R(M,C),C))\\
&=l_R(\H_{-t}(\uhom_R(\H_{-t}(\uhom_R(M,C)),C))\\
&\leq l_R(\H_{-t}(\uhom_R(M,C))-s\\
&\leq l_R(M)-2s.
\end{array}\]
Therefore, $s=0$ and so $\H_{-\depth_R(C)-1}(\uhom_{R}(k,C))=0$. Now, by the exact sequence (3.6) one can deduce that $\H_{-\depth_R(C)-1}(\uhom_{R}(M_2,C))=\dots =\H_{-\depth_R(C)-1}(\uhom_{R}(M_{n-1},C))=0$. Hence one has $\H_{-\depth_R(C)-2}(\uhom_{R}(k,C))\cong\H_{-\depth_R(C)-1}(\uhom_{R}(M_{n-1},C))$, and thus $\H_{-\depth_R(C)-2}(\uhom_{R}(k,C))=0$. Therefore, by similar argument, one can deduce that $\H_{-j}(\uhom_{R}(k,C))=0$ for all $j>\depth_R(C)=-\inf C$, as required.
\end{proof}
\end{thm}

%%%%%%%%%%%%%%%%%%%%%%%%%%%%%%%%%%%%%%%%%%%%%%%%%%%%%%%%%%%%%%%%%%%%%%%%%%%%%%%%%%%%%%%%%%%%%%%%%%%%%%%%%%%%%%%%%%%%%%%%%%%%%%%%%%%%%%%%%%%%%%%%%%%%%%%%%%%%%%%%%%%%
%%%%%%%%%%%%%%%%%%%%%%%%%%%%%%%%%%%%%%%%%%%%%%%%%%%%%%%%%%%%%%%%%%%%%%%%%%%%%%%%%%%%%%%%%%%%%%%%%%%%%%%%%%%%%%%%%%%%%%%%%%%%%%%%%%%%%%%%%%%%%%%%%%%%%%%%%%%%%%%%%%%%
The following result, which is a consequence of the previous theorem, provides a characterization of dualizing modules of $R$ and a generalization of \cite[Theorem 2.3]{Tak}.
\begin{cor}\label{COOO}Let $(R,\fm)$ be a local ring and let $C$ be a semidualizing $R$-module. Then the following statements are equivalent:
\begin{itemize}
\item[(i)]{$C$ is dualizing}.
\item[(ii)]{$R$ admits a Cohen-Macaulay $R$-module of type 1 and finite $G_{C}$-dimension.}
\item[(iii)]{$R$ is a ring of type $\beta_0^R(C)$ admitting a Cohen-Macaulay $R$-module of finite $G_{C}$-dimension.}
\end{itemize}
\begin{proof}The implication (i)$\Longrightarrow$(ii) is clear and the implication (ii)$\Longrightarrow$(iii) follows from Lemma \ref{Th3}. For the implication (iii)$\Longrightarrow$(i), let $r_R(R)=\beta_{0}^R(C)$ and let $M$ be a Cohen-Macaulay $R$-module of finite $G_{C}$-dimension. Hence, in view of Lemma \ref{Th3} one has $\r_R(C)=1$, and so Theorem \ref{Maa} implies that $C$ is dualizing.
\end{proof}
\end{cor}

%%%%%%%%%%%%%%%%%%%%%%%%%%%%%%%%%%%%%%%%%%%%%%%%%%%%%%%%%%%%%%%%%%%%%%%%%%%%%%%%%%%%%%%%%%%%%%%%%%%%%%%%%%%%%%%%%%%%%%%%%%%%%%%%%%%%%%%%%%%%%%%%%%%%%%%%%%%%%%%%%%%%
%%%%%%%%%%%%%%%%%%%%%%%%%%%%%%%%%%%%%%%%%%%%%%%%%%%%%%%%%%%%%%%%%%%%%%%%%%%%%%%%%%%%%%%%%%%%%%%%%%%%%%%%%%%%%%%%%%%%%%%%%%%%%%%%%%%%%%%%%%%%%%%%%%%%%%%%%%%%%%%%%%%%

Let $C$ be a semidualizing complex and $X$ be a homologically finitely generated and bounded complex. Then we define, $\gr_C(X)$, garde of $X$ with respect to $C$ as follow: $\gr_C (X):=\inf\{ i | \Ext_R^i(X,C)\neq 0\}=-\sup\uhom_R(X,C)$. Now, we provide the following result which is a generalization of \cite[Corollary 4.4]{DFT}.
\begin{prop}{\label{Ma}}Let $(R,\fm)$ be a local ring and let $C$ be a semidualizing $R$-complex. Then the following statements are equivalent:
\begin{itemize}
\item[(i)]{$C$ is Cohen-Macaulay}.
\item[(ii)]{$R$ admits a Cohen-Macaulay complex $X$ with finite $G_{C}$-dimension such that $\dim_R(X)=\dim_R(C)-\gr_C(X)$.}
\end{itemize}
\begin{proof}For the implication (i)$\Longrightarrow$(ii), first notice that by definition of semidualizing complex, $\GCd_R C$ is finite and $\gr_C(C)=0$. Therefore, by setting  $X=C$, there is nothing to prove. The implication (ii)$\Longrightarrow$(i) follows from \cite[Lemma 3.1]{CF1} and \cite[Theorem 3.3(a)]{CF1} which say that $\sup\uhom_R(X,C)=\inf\uhom_R(X,C)=\depth_R(X)-\depth_R(C).$
\end{proof}
\end{prop}
%%%%%%%%%%%%%%%%%%%%%%%%%%%%%%%%%%%%%%%%%%%%%%%%%%%%%%%%%%%%%%%%%%%%%%%%%%%%%%%%%%%%%%%%%%%%%%%%%%%%%%%%%%%%%%%%%%%%%%%%%%%%%%%%%%%%%%%%%%%%%%%%%%%%%%%%%%%%%%%%%%%
%%%%%%%%%%%%%%%%%%%%%%%%%%%%%%%%%%%%%%%%%%%%%%%%%%%%%%%%%%%%%%%%%%%%%%%%%%%%%%%%%%%%%%%%%%%%%%%%%%%%%%%%%%%%%%%%%%%%%%%%%%%%%%%%%%%%%%%%%%%%%%%%%%%%%%%%%%%%%%%%%%%

\begin{thm}{\label{Maaaa}}Let $(R,\fm)$ be a local ring and let $C$ be a semidualizing $R$-complex and consider the following statements:
\begin{itemize}
\item[(i)]{$C$ is dualizing}.
\item[(ii)]{$R$ admits a Cohen-Macaulay complex of type 1 with finite $G_{C}$-dimension.}
\item[(iii)]{$\r_R(R)=\beta_{\inf C}^R(C)$ and there exists a Cohen-Macaulay $R$-complex $X$ of finite $G_{C}$-dimension.}
\item[(iv)]{$\r_R(C)=1$ and there exists a Cohen-Macaulay $R$-complex $X$ of finite $G_{C}$-dimension.}
\end{itemize}
Then the implications \emph{(i)$\Longrightarrow$(ii), (ii)$\Longrightarrow$(iii) and (iii)$\Longleftrightarrow$(iv)} hold and the implication \emph{(iv)$\Longrightarrow$(i)} holds in the each of the following cases:
\begin{itemize}
\item[(1)]{$\amp X=0$.}
\item[(2)]{$X$ is a dualizing complex.}
\item[(3)]{ $X$ satisfies the condition $\dim_R(X)=\dim_R(C)-\gr_C(X)$.}
\end{itemize}
\begin{proof}
The implications (i)$\Longrightarrow$(ii) follows from Theorem \ref{Anni} and the fact that $\GCd_{R}C$ is finite. The implications (ii)$\Longrightarrow$(iii) and  (iii)$\Longleftrightarrow$(iv) follow from Lemma \ref{Th3}. Now assume that \emph{$\r_R(C)=1$} and $X$ is a Cohen-Macaulay $R$-complex with finite $G_C$-dimension. Hence, the implication (iv)$\Longrightarrow$(i) in the case (1) follows from Theorem \ref{Maa}, and in the case (2) one can see \cite[Theorem 8.2($ii^{,}$)]{CF1}. Also, in case (3), it follows from Proposition \ref{Ma} and Theorem \ref{Anni}.
\end{proof}
\end{thm}
Here we should notice that, in general, using Theorem \ref{Anni} and Lemma \ref{Th3} implies that to prove the implication (iv)$\Longrightarrow$(i), in the above theorem, it is enough to show that $C$ is Cohen-Macaulay. Furthermore, in view of \cite[Proposition 3.15]{CF1} for an $R$-complex $X\in\mathrm{D}_{\Box}^f(R)$ one has the inequality $\GCd_R(X)\leq\pd_R(X)$ which is equality if $\pd_R(X)$ is finite. On the other hand, by \cite[Proposition 3.10]{CF1}, $R$ is Cohen-Macaulay whenever there exists a Cohen-Macaulay $R$-complex of finite projective dimension. So, it is natural to ask the following question.

%%%%%%%%%%%%%%%%%%%%%%%%%%%%%%%%%%%%%%%%%%%%%%%%%%%%%%%%%%%%%%%%%%%%%%%%%%%%%%%%%%%%%%%%%%%%%%%%%%%%%%%%%%%%%%%%%%%%%%%%%%%%%%%%%%%%%%%%%%%%%%%%%%%%%%%%%%%%%%%%%%%
%%%%%%%%%%%%%%%%%%%%%%%%%%%%%%%%%%%%%%%%%%%%%%%%%%%%%%%%%%%%%%%%%%%%%%%%%%%%%%%%%%%%%%%%%%%%%%%%%%%%%%%%%%%%%%%%%%%%%%%%%%%%%%%%%%%%%%%%%%%%%%%%%%%%%%%%%%%%%%%%%%%
\begin{ques}Let $(R,\fm)$ be a local ring and let $C$ be a semidualizing $R$-complex. If $R$ admits a Cohen-Macaulay complex with a positive amplitude of finite $G_{C}$-dimension and $r_R(C)=1$, then can we say that $C$ is Cohen-Macaulay?
\end{ques}
%%%%%%%%%%%%%%%%%%%%%%%%%%%%%%%%%%%%%%%%%%%%%%%%%%%%%%%%%%%%%%%%%%%%%%%%%%%%%%%%%%%%%%%%%%%%%%%%%%%%%%%%%%%%%%%%%%%%%%%%%%%%%%%%%%%%%%%%%%%%%%%%%%%%%%%%%%%%%%%%%%
%%%%%%%%%%%%%%%%%%%%%%%%%%%%%%%%%%%%%%%%%%%%%%%%%%%%%%%%%%%%%%%%%%%%%%%%%%%%%%%%%%%%%%%%%%%%%%%%%%%%%%%%%%%%%%%%%%%%%%%%%%%%%%%%%%%%%%%%%%%%%%%%%%%%%%%%%%%%%%%%%%

From \cite{RHart} we know that, up to isomorphism and shift, a dualizing complex $D$ is the only semidualizing complex of finite injective dimension. Also, in \cite[Proposition 8.3]{CF1}, it has been proven that up to isomorphism and shift $R$ is the unique semidualizing complex in $\mathcal{R}_R(R)$, in particular, it is the only semidualizing complex of finite projective dimension. So, the following result can be considered as an improvement of \cite[Proposition 8.3]{CF1}.
\begin{thm}\label{fai}Let $(R,\fm, k)$ be local ring and $C$ be a semidualizing $R$-complex. Then the following statements are equivalent:
 \begin{itemize}
\item[(i)]{$C\sim R$}.
\item[(ii)]{$k\in\mathcal{A}_C(R)$}.
\item[(iii)]There exists a Cohen-Macaulay $R$-module $M$ of type one such that $M\in\mathcal{A}_C(R)$.
\end{itemize}
\begin{proof}First notice that for an $R$-complex $X\in\mathrm{D}_{\Box}^f(R)$ one can deduce that $X\in\mathcal{A}_C(R)$ if and only if $X\otimes_R\widehat{R}\in\mathcal{A}_{C\otimes_R\widehat{R}}(\widehat{R})$. Furthermore, there exist the equalities $\dim_R(X)=\dim_{\widehat{R}}(X\otimes_R\widehat{R})$, $\depth_R(X)=\depth_{\widehat{R}}(X\otimes_R\widehat{R})$, $\r_R(X)=\r_{\widehat{R}}(X\otimes_R\widehat{R})$ and $\gr_C (X)=\gr_{C\otimes_R\widehat{R}}(X\otimes_R\widehat{R})$. Therefore, we may and do assume that $R$ is a complete local ring and so has a dualizing complex as $D$. Now, the implication (i)$\Longrightarrow$(ii) follows from the fact that $\mathcal{A}_R(R)=\mathrm{D}_{\Box}(R)$, see \cite[Remmark 4.3]{CF1}. The implication (ii)$\Longrightarrow$(iii) is clear because $k$ is a Cohen-Macaulay $R$-module of type one. For the implication (iii)$\Longrightarrow$(i), first we should notice that in view of \cite[Theorem 4.7]{CF1} one has $\uhom_R(C,D)$ is a semidualizing complex and $\GCad_R(M)$ is finite. Hence, by Lemma \ref{Th3}, we conclude that $\uhom_R(C,D)$ has type one and so it is a dualizing complex of $R$ by Theorem \ref{Maa}. Therefore, there exists an integer $n$ such that $\uhom_R(C,D)\simeq\Sigma^nD$. Hence, one has the following isomorphism:
\[\begin{array}{rl}
C&\simeq\uhom_R(\uhom_R(C,D),D)\\
&\simeq \uhom_R(\Sigma^nD,D)\\
&\simeq \Sigma^{-n}\uhom_R(D,D)\\
&\simeq \Sigma^{-n}R,
\end{array}\]
as required.

\end{proof}
\end{thm}

Acknowledgments. The author thanks Prof. Kamaran Divaani-Aazar for having a useful conversation with him during the preparation of the manuscript.
%%%%%%%%%%%%%%%%%%%%%%%%%%%%%%%%%%%%%%%%%%%%%%%%%%%%%%%%%%%%%%%%%%%%%%%%%%%%%%%%%%%%%%%%%%%%%%%%%%%%%%%%%%%%%%%%%%%%%%%%%%%%%%%%%%%%%%%%%%%%%%%%%%%%%%%%%%%%%%%%%%%
%%%%%%%%%%%%%%%%%%%%%%%%%%%%%%%%%%%%%%%%%%%%%%%%%%%%%%%%%%%%%%%%%%%%%%%%%%%%%%%%%%%%%%%%%%%%%%%%%%%%%%%%%%%%%%%%%%%%%%%%%%%%%%%%%%%%%%%%%%%%%%%%%%%%%%%%%%%%%%%%%%%

\end{document}